\documentclass{amsart}
\usepackage{graphicx} 
\usepackage{amsmath}
\usepackage{amsfonts}
\usepackage{amssymb}
\usepackage{amsthm}
\usepackage{hyperref}
\usepackage{listings}
\usepackage{tikz-cd}
\usepackage{enumitem}
\usepackage[alphabetic,initials]{amsrefs}
\usepackage{mathrsfs}  

\newcommand{\mc}{\mathscr}
\newcommand{\PP}{\mathbb{P}}
\newcommand{\QQ}{\mathbb{Q}}
\newcommand{\RR}{\mathbb{R}}
\newcommand{\CC}{\mathbb{C}}
\newcommand{\ZZ}{\mathbb{Z}}
\newcommand{\git}{/\!\!/}

\renewcommand{\L}[1]{\mc{L}_{\mathrm{#1}}}
\newcommand{\Sfc}[1]{S_{\mathrm{#1}}}
\newcommand{\Lone}{\L{V}}
\newcommand{\Rone}{\mc R(\Lone)}

\newcommand{\Sbar}{\widetilde{S}}
\newcommand{\pibar}{\widetilde{\pi}}
\newcommand{\E}{E}
\newcommand{\Q}{ Q}

\renewcommand{\i}{\mathbf{i}}

\title{Line configurations and K3 surfaces}
\author{Elias Sink}
\address{Lederle Graduate Research Tower, 1654
University of Massachusetts Amherst
710 N. Pleasant Street
Amherst, MA 01003-9305, USA}
\curraddr{Harvard University
Department of Mathematics
Science Center Room 325
1 Oxford Street
Cambridge, MA 02138
USA}
\email{esink@math.harvard.edu}
\date{\today}
\subjclass[2020]{Primary 14N20, 14J28, 14G05; Secondary 14J27}
\thanks{The author been partially supported by the NSF grant DMS-2101726 (PI Jenia Tevelev).}

\numberwithin{equation}{section}

\newtheorem{thm}{Theorem}[section]
\newtheorem{prp}[thm]{Proposition}
\newtheorem{lem}[thm]{Lemma}
\newtheorem{cor}[thm]{Corollary}

\theoremstyle{definition}
\newtheorem{dfn}[thm]{Definition}
\newtheorem{rem}[thm]{Remark}
\newtheorem{con}[thm]{Convention}
\newtheorem{ntn}[thm]{Notation}
\newtheorem{comp}[thm]{Computation}

\begin{document}

\begin{abstract}
We study the realization spaces of $10_3$ line configurations. Answering a question posed by Sturmfels in 1991, we use elliptic surface techniques to show that realizations over $\mathbb{Q}$ are dense in those over $\mathbb{R}$ for all $10_3$ configurations. We find that for exactly four of the ten configurations, the realization space admits a compactification by a K3 surface. We show that these have Picard number 20 and compute their discriminants. Finally, we use geometric invariant theory to give an elegant interpretation of these K3 surfaces as moduli spaces.
\end{abstract}

\maketitle

\section{Introduction}
A \emph{configuration} on a finite set $\mc{P}$ (called the set of \emph{points}) is a set $\mc{L}$ of subsets of $\mc{P}$ (called \emph{lines}) such that any two lines have at most one point in common. A \emph{realization} of $\mc{L}$ in $\PP^2$ over a field $k$ is a map $\mc{P}\to \PP^2(k)$ such that for all distinct $p,q,r\in \mc{P}$, their images in $\PP^2(k)$ are collinear if and only if there is a line in $\mc{L}$ containing $p,q,r$. Configurations and their realizations over various fields have been studied extensively in both theoretical and applied contexts since the late nineteenth century. Especially fascinating are so-called $n_3$ configurations, which have $n$ points and $n$ lines such that each line contains exactly three points and each point lies on three lines. All such configurations for $7\leq n\leq 12$ have been tabulated, and it is known which ones are realizable over $\QQ$ \cite{gropp}. 

The set of all realizations over $k$ of a configuration $\mc{L}$ can be identified with the $k$-points of a quasiprojective variety $V\subset (\PP^2)^n$, where $n=|\mc{P}|$. The condition that three points are collinear is expressed as the vanishing of a corresponding minor of the $3\times n$ matrix of coordinates of $(\PP^2)^n$, and $V$ is cut out by one equation or inequation of this type for each triple. However, this variety is too large; we would like to identify realizations that are related by projective transformations of $\PP^2$. This quotient space $V/\mathrm{PGL}(3)$ is called the \emph{realization space} of the configuration, denoted $\mc{R}(\mc L)$. This quotient is constructed by means of geometric invariant theory, and the choice of stability condition gives rise to various compactifications of~$\mc{R}(\mc L)$, each with its own interpretation as a moduli space of (weak) realizations of~$\mc L$ (see Section~\ref{sec:mod}).

Specializing to the case of $n_3$ configurations, we observe that the equation defining a line $\ell\in \mc L$ has degree $1$ in the coordinates of the points on $\ell$, and degree $0$ for other points. Since each point lies on exactly three lines, the sum of the multidegrees of the $n$ equations of lines in $\mc L$ is $(3,3,\dots,3)$. This is exactly opposite the multidegree of the canonical class of $(\PP^2)^n$. If the projective variety $W\supset V$ cut out by these $n$ equations were a complete intersection of codimension $n$ in $(\PP^2)^{n}$, the adjunction formula would imply that the canonical bundle of $W$ is trivial. The same argument applies in any open subset of $(\PP^2)^n$; in particular, if the semistable locus of $W$ is such a complete intersection in $(\PP^2)^n_{ss}$, then the corresponding GIT quotient will have trivial canonical bundle as well (see Lemma~\ref{lem:can}). In other words, we expect the realization spaces of $n_3$ configurations to have compactifications with ``Calabi--Yau type" geometry. Such varieties are of significant interest for their difficult arithmetic and their relevance to physics. For example, in the $10_3$ case, $\mc{R}(\mc L)$ has expected dimension $$2n-|\mc L|-\dim \mathrm{PGL}(3)=20-10-8=2,$$ so we would hope for a K3 surface. Similarly, $11_3$ configurations would give Calabi--Yau threefolds, and so on.

As a test of this philosophy, we study the realization spaces of $10_3$ configurations. It was shown by Kantor \cite{kantor} that there are precisely $10$ such configurations (up to relabeling), and Schroeter \cite{schroeter} found that all but one of them admit realizations over fields of characteristic $0$. Following Schroeter's numbering, we refer to these configurations as $\L{I},\L{II},\dots,\L{X}$. To illustrate our methods, we focus our discussion on
\begin{equation} \label{eq:lone}
\L{V}=\{124,138,179,237,259,350,456,480,678,690\},
\end{equation}
where $\mc{P}=\{0,\dots,9\}$ and $124=\{1,2,4\}$, etc. This configuration was studied by Sturmfels in \cite{sturmfels}, who gave a concrete description of its realization space using a geometric construction sequence. Our choice of this particular configuration was motivated by the question, left open by Sturmfels, of whether its realizations over $\QQ$ are dense in the realizations over $\RR$. We give a positive answer to this question.

\begin{thm}\label{thm:dense}
    The rational realizations $\mc R(\mc L)(\QQ)$ are Zariski-dense in $\mc R(\mc L)(\RR)$ for all $10_3$ configurations $\mc L$, and dense in the analytic topology for $\mc L\neq \L{X}$.
\end{thm}

(We do not know whether analytic density holds for $\L{X}$; see Remark~\ref{rem:X}.) It turns out that for many $10_3$ configurations, the would-be K3 surfaces are either not of the expected dimension ($\L{I}$) or are reducible ($\L{II}$, $\L{III}$, $\L{IV}$, $\L{VI}$, and $\L{VII}$). For the other four, the Calabi--Yau dream is achieved.

\begin{thm}\label{thm:k3moduli} Let $\mc L$ be one of the $10_3$ configurations $\L{V}$, $\L{VIII}$, $\L{IX}$, or $\L{X}$. 
    \begin{enumerate}[label=(\roman*),ref=\thethm(\roman*)]
        \item \label{part:k3} The realization space $\mc R(\mc L)$ is isomorphic to a Zariski-open subset of an elliptic K3 surface of Picard number $20$ and discriminant $-7$, $-8$, $-7$, and $-11$, respectively.
        \item \label{part:moduli} This K3 surface is a fine moduli space for GIT-stable weak realizations of~$\mc L$.
    \end{enumerate}
\end{thm}

In Section~\ref{sec:ell}, we observe that Sturmfels' calculation gives rise to an algebraic surface $S$ with an elliptic fibration. We use computations in its Mordell--Weil group to prove Theorem~\ref{thm:dense} for $\mc L=\L{V}$. Similar techniques are used for the other configurations. In Section~\ref{sec:k3}, we construct a K3 surface $\Sbar$ as the minimal resolution of $S$. We compute its singular fibers, Picard number, and discriminant, which identify it as the universal elliptic curve over $\Gamma_1(7)$. The other three K3 surfaces are constructed likewise, proving Theorem~\ref{part:k3}. In Section~\ref{sec:mod}, we review GIT quotients in general and for the case of $(\PP^2)^{n}\git\mathrm{PGL}(3)$. We describe the correct choice of GIT quotient for our problem and use computer algebra to prove Theorem~\ref{part:moduli}.

Throughout, we work with varieties over $\CC$, though our results (except for those on analytic density) hold over any algebraically closed field of characteristic $0$. 

\subsection*{Acknowledgements} The author thanks Jenia Tevelev for his inspiring mentorship and for suggesting this project, Aditya Khurmi for many useful discussions and help with the creation of Figure~\ref{fig:ell} below, as well as Alejandro Morales and Andreas Buttensch\"on for organizing the Summer 2023 REU program at the University of Massachusetts Amherst, during which much of this research was conducted. We make essential use of the computer algebra systems Magma \cite{Magma} and Macaulay2~\cite{M2} (see Appendix \ref{sec:comp}).
\section{Elliptic fibrations and density of rational realizations} \label{sec:ell}
Sturmfels' parameterization of $\Rone$ goes as follows (see \cite{sturmfels}*{\S 2} for details): fix points $1$, $2$, $3$, and $5$ to standard coordinates $[1:0:0]$, $[0:1:0]$, $[0:0:1]$, and $[1:1:1]$ using a projective transformation. Let $[a:b:c]$ be homogeneous coordinates for point $6$, and $[0:u:v]$ homogeneous coordinates for point $7$ along the line $\overline{23}$ (both in ``general position" to avoid unwanted collinearities). The positions for the other $4$ points are then determined, with one condition to ensure that $6,9,0$ are collinear:
\begin{equation}\label{eq:sturmfels}
    u^2a^2c-uva^2c-v^2b^3+uvb^2c+v^2ab^2-uvabc+uvac^2-u^2ac^2=0.
\end{equation}
It's clear that every realization of $\Lone$ (up to $\mathrm{PGL}(3)$) can be obtained in this manner for a unique choice of $([u:v],[a:b:c])$ satisfying \eqref{eq:sturmfels}, and that almost all choices yield such a realization. Sturmfels also gives an example of a realization of $\Lone$ over $\QQ$, so $\Rone(\QQ)$ is nonempty.

We observe that equation \eqref{eq:sturmfels} is irreducible and homogeneous of bidegree $(2,3)$ in $[u:v]$ and $[a:b:c]$, so it defines an irreducible surface $S\subset\PP^1\times \PP^2$. The above analysis shows that $\Rone$ is isomorphic (over $\QQ$) to a dense open subset of $S$. Moreover, Computation~\ref{comp:singularpoints} shows that this subset is contained in the smooth locus $S_{sm}$ of $S$. Hence, to prove Theorem~\ref{thm:dense} for $\mc L=\L{V}$, it suffices to show that $S_{sm}(\QQ)$ is analytically dense in $S_{sm}(\RR)$.

The generic fiber $\E$ of the projection $\pi:S\to \PP^1$ onto the first factor is a smooth cubic plane curve over the function field $K=k(\PP^1)$ whose $K$-points are identified with the sections of $\pi$. That is, $\pi$ gives an \emph{elliptic fibration} of $S$. The point $([u:v],[1:1:1])$ lies in $S$ for all $u:v\in\PP^1$, so we have a section $$o=[1:1:1]\in\E(K)$$ defined over $\QQ$. Choosing $o$ for the identity makes $\E$ into an elliptic curve over $K$. 

The abelian group $\mathrm{MW}(S)=\E(K)$ of sections of $\pi$ is called the (geometric) Mordell--Weil group of $S$. This group is finitely generated \cite{schuttshioda}*{Theorem 6.1} and amenable to computer calculations, which will allow us to show the existence of many rational points. As an easy demonstration of this technique, we prove the following:

\begin{prp} \label{prp:zariski}
$S(\QQ)$ is dense in $S$ with respect to the Zariski topology.
\end{prp}  
\begin{proof}
We find another section $$s=[1:t:t]\in E(K)$$ defined over $\QQ$, where $t=u/v\in K$. We check in Computation~\ref{comp:mordellweil} that $s$ is not torsion in $\mathrm{MW}(S)$, so $\pi$ admits infinitely many sections defined over $\QQ$. That is, $S$ contains an infinite collection of irreducible curves isomorphic to $\PP^1$ over $\QQ$. Their union is certainly Zariski-dense; a proper closed subset of $S$ must have codimension at least $1$ by irreducibility, and therefore can contain only finitely many irreducible curves. Since each section has a dense set of $\QQ$-points, this proves that $S(\QQ)$ is Zariski-dense.
\end{proof}

The proof of analytic density proceeds along similar lines. 

\begin{con}
For the remainder of this section only, we use the analytic topology when working with real loci. In particular, ``dense", ``open", and ``connected" are understood with respect to this topology.
\end{con}

\begin{lem}\label{lem:orbits}
Let $E$ be an elliptic curve over $\RR$, and let $p$ be an non-torsion $\RR$-point. Then the orbit of $p$ is dense in $E(\RR)$ if and only if either $E(\RR)$ is connected, or $p$ does not lie in the identity component of $E(\RR)$. 
\end{lem}

\begin{proof}
The real locus of an elliptic curve has either one or two connected components, both diffeomorphic to circles. Being a compact connected real Lie group of dimension $1$, the identity component $E_0$ of $E(\RR)$ is a normal subgroup isomorphic (as a Lie group) to the circle group $\RR/\ZZ$. If $\E(\RR)=E_0$, then the claim follows from the standard fact that irrational rotations have dense orbits in $\RR/\ZZ$. 

Otherwise, $E(\RR)$ has two components, $E_0$ and $E_1$, with the quotient $E(\RR)/E_0\cong \ZZ/2$. If $p\in E_0$, then its orbit is contained in $E_0$ and not dense in $E(\RR)$. Otherwise, we have $E_1=p+E_0$ and $2p\in E_0$. The orbit of $2p$ is dense in $E_0$ by the above, and its image under translation by $p$ is dense in $E_1$.
\end{proof}

\begin{lem} \label{lem:dense}
Suppose $f: X\to \PP^1$ is an elliptic fibration over $\QQ$ with identity section $o$. Let $F(t)$ denote the real locus of the fiber $f^{-1}(t)$ over $t\in\PP^1(\RR)$. Suppose further that there are open sets $U_i\subseteq \PP^1(\RR)$ and sections $s_i\in \mathrm{MW}(X)$ defined over $\QQ$ such that
\begin{enumerate}[label=(\arabic*),ref=(\arabic*)]
    \item \label{part:union}  $\bigcup_i U_i$ is dense in $\PP^1(\RR)$, 
    \item \label{part:order} the $s_i$ are not torsion in $\mathrm{MW}(X)$, and
    \item \label{part:comp} for all $t\in U_i$ with $F(t)$ smooth, either $F(t)$ is connected or $s_i(t)$ lies in the non-identity component of $F(t)$. 
\end{enumerate}
Then $X_{sm}(\QQ)$ is dense in $X_{sm}(\RR)$, where $X_{sm}$ is the smooth locus of $X$.
\end{lem}
\begin{proof}
When $F(t)$ is smooth, we regard it as a real elliptic curve with identity $o(t)$. As a curve in $X$, any positive multiple of $s_i$ in $\mathrm{MW}(X)$ intersects the identity section in finitely many points, so the set 
$$T_{i,n}=\{t\in \PP^1(\QQ)\mid s_i(t) \text{ has order } n \text{ in } F(t)\}$$ is finite. By Mazur's classification of torsion subgroups of elliptic curves over $\QQ$ \cite{mazur}, $T_{i,n}$ can only be nonempty if $n\leq 12$, so $T_i=\bigcup_{n>0} T_{i,n}$ is finite. It follows that the set $$D_i=(U_i\cap \PP^1(\QQ))\smallsetminus T_i,$$ where $s_i(t)$ is not torsion in $F(t)$, is dense in $U_i$. By assumption~\ref{part:comp} and Lemma~\ref{lem:orbits}, the orbit of $s_i(t)$ (which consists of rational points) is dense in $F(t)$ for all $t\in D_i$. 

Suppose $W\subseteq X_{sm}(\RR)$ is open and nonempty. $X_{sm}(\RR)$ is a real manifold of dimension $2$, so nonempty Zariski-open sets are (analytically) dense. In particular, the differential $\mathrm{d}f$ is nonzero on such a subset, so $f$ is a submersion on some nonempty open subset of $W$. Submersions are open maps, so $f(W)\subseteq \PP^1(\RR)$ contains an open set. By~\ref{part:union}, $f(W)$ intersects some $U_i$, and hence some $D_i$. Then $W$ meets meets $F(t)$ for some $t\in D_i$. $F(t)$ has dense rational points, so the proof is complete.
\end{proof}

\begin{cor} \label{cor:dense}
With notation as above, suppose there exist $s,r\in \mathrm{MW}(X)$ defined over $\QQ$ such that $s$ is non-torsion, $r$ is torsion, and for all $t\in \PP^1(\RR)$ with $F(t)$ smooth and disconnected, $r(t)$ lies in the non-identity component of $F(t)$. Then $X_{sm}(\QQ)$ is dense in $X_{sm}(\RR)$.
\end{cor}

\begin{proof}
Let $U=\{t\in \PP^1(\RR)\mid F(t) \text{ is smooth}\}$, which is a dense open subset of $\PP^1(\RR)$. Let $$U_0=\left\{t\in U\mid \substack{F(t)\text{ is connected or } s(t) \text{ lies in}\\ \text{the non-identity component of } F(t)}\right\}$$
and $U_1=U\smallsetminus U_0$. Both are easily seen to be open. The result follows from Lemma~\ref{lem:dense} with $s_0=s$ and $s_1=s+r$.
\end{proof}
\begin{rem}\label{rem:finite}
By continuity, it suffices to check the condition on $r(t)$ for one $t$ in each connected component of $U$. Roughly speaking, $r(t)$ cannot jump between components except when $F(t)$ degenerates. 
\end{rem}
\begin{lem} \label{lem:lv}
Theorem~\ref{thm:dense} holds for $\mc L = \L{V}$.
\end{lem}
\begin{proof} We first compute a generalized Weierstrass form for $\E$. Computation~\ref{comp:mordellweil} gives
\begin{equation}\label{eq:weierstrass}
    y^2=x^3+a(t)x^2+b(t)x
\end{equation}
where
$$t=u/v\in K,\quad a(t)=\frac{8t^3 - 15t^2 + 8t}{(t-1)^2},\quad b(t)=16t^2.$$
(For convenience, we use affine coordinates throughout the proof.) The surface $S'\subset \PP^1\times \PP^2$ given by equation \eqref{eq:weierstrass} in $(t,x,y)$ is birational to $S$ over $\QQ$. It therefore suffices to show that $S'_{sm}(\QQ)$ is dense in $S'_{sm}(\RR)$.

As with $S$, we have an elliptic fibration $\pi':S'\to \PP^1$. We study the real locus $F(t)$ of the fiber over $t\in U= \PP^1(\RR)\smallsetminus \{0,1,\infty\}$ as a real curve in $\RR^2$ with coordinates $(x,y)$. $F(t)$ is smooth for $t\in U$, and it has two components if and only if $x^2+a(t)x+b(t)$ has distinct real roots, i.e., 
$$a(t)^2-4b(t)=\frac{t^3}{(t-1)^4}(16t^2 - 31t + 16)>0.$$
The quadratic factor is strictly positive, so $F(t)$ has two components exactly when $t>0$. Furthermore, $a(t)$ and $b(t)$ are positive for $t>0$, so 
\begin{equation}\label{eq:idcomp}
    \left(x,\pm\sqrt{x^3+a(t)x^2+b(t)x}\right)
\end{equation} 
is a real point for any $x\geq 0$. Since $x^3+a(t)x^2+b(t)x$ has a root at $0$, the identity component of the fiber over $t>0$ is exactly \eqref{eq:idcomp} for $x\geq 0$.

We now apply Corollary~\ref{cor:dense}. We saw in the proof of Proposition~\ref{prp:zariski} that $E(K)$ has a non-torsion element $s$ defined over $\QQ$. Computation~\ref{comp:mordellweil} exhibits the torsion section
$$r = \left(-4t,\frac{4t^2}{t-1}\right),$$
defined over $\QQ$. Moreover, it lies in the non-identity component of $F(t)$ when $t>0$, so the proof is complete.
\end{proof}
\begin{proof}[Proof of Theorem~\ref{thm:dense}]
Following \cite{sturmfels}, we construct the realization spaces of the other $10_3$ configurations in the same fashion: choose four points $p_1,\dots,p_4\in \mc P$, no three on a line in $\mc L$, such that $p_1,p_2\in\ell$ for some $\ell\in\mc L$. Fix these points to $[1:0:0]$, $[0:1:0]$, $[0:0:1]$, and $[1:1:1]$. Let the third point on $\ell$ be $[u:v:0]$, and let $[a:b:c]$ be another point in general position with the first five. Nine lines suffice to determine the positions of the remaining points, and the tenth line gives a bihomogeneous equation $F_{\mc L}([u:v],[a:b:c])=0$ of bidegree $(2,3)$. Imposing open conditions to exclude additional collinearities presents the realization space as a subvariety of $\PP^1\times \PP^2$. This construction is carried out for each configuration in Computations~\ref{comp:other} and ~\ref{comp:k3}.

Two configurations are special. One is $\L{I}$, the well-known Desargues configuration. Here, the nine collinearities imply the tenth (Desargues' theorem), so $F_{\L{I}}$ is identically $0$. $\mc R(\L{I})$ is a Zariski-open subset of $\PP^1\times \PP^2$, so it has dense rational points. The other is the unique non-realizable $10_3$ configuration $\L{IV}$, for which the claim is vacuous. 

For all other configurations, the realization space is a surface. For configurations $\L{II}$, $\L{III}$, $\L{VI}$, and $\L{VII}$, the surface given by $F_{\mc L}=0$ is reducible. In each case, all but one component are eliminated by the open conditions, leaving a smooth rational surface as the realization space. Again, density is immediate. 

For the remaining configurations $\L{V}$ (treated above), $\L{VIII}$, $\L{IX}$, and $\L{X}$, the equations $F_{\mc L}=0$ define irreducible surfaces $\Sfc{V}=S$, $\Sfc{VIII}$, $\Sfc{IX}$, and $\Sfc{X}$ in $\PP^1\times \PP^2$ with elliptic fibrations given by projection to $\PP^1$. We use a series of computations, collected in Computation~\ref{comp:k3} below, to complete the proof. To prove Zariski density, it suffices to find a non-torsion element $s$ of each Mordell--Weil group (see Proposition~\ref{prp:zariski}). For analytic density, we observe that the generic fibers of $\Sfc{V}$ and $\Sfc{IX}$ are isomorphic (cf. Table~\ref{tab:fib} below), so it remains to study $\Sfc{VIII}$. We proceed as in Lemma~\ref{lem:lv}: We first check that the singularities of the surface are disjoint from the realization space, and we compute the Weierstrass form $y^2=x^3+a(t)x^2+b(t)x$ where $a(t)=4t(t-1)^2$ and $b(t)=-16t^3(t-1)^2$. In addition to the non-torsion section $s$, we have a torsion section $r$ given by $(x,y)=(0,0)$. When $1\neq t>0$, we have $b(t)<0$, so $x=0$ is the middle root of the cubic; when $t<0$, we have $b(t)>0$ and $a(t)<0$, so $x=0$ is the smallest root. Thus $r$ lies in the nonidentity component of every smooth fiber, and the claim follows from Corollary~\ref{cor:dense}.
\end{proof}
\begin{rem}\label{rem:X}
    This method does not show analytic density of $\Sfc{X}(\QQ)$, because there is an interval on the base over which \textit{every} section defined over $\QQ$ intersects the fibers in their identity component; see Computation~\ref{comp:k3} for details. 
\end{rem}
\section{The K3 surfaces}\label{sec:k3}
In this section, we study the surface $S=\Sfc{V}$ from Section~\ref{sec:ell} and its elliptic fibration $\pi$ in greater detail. While the general fiber of an elliptic fibration is a smooth curve of genus $1$, there may be finitely many points where the fiber degenerates into something singular. Our fibration $\pi$ has five such singular fibers. The fibers ${\pi^{-1}([0:1])}$, ${\pi^{-1}([1:1])}$, and ${\pi^{-1}([1:0])}$ are unions of lines in $\PP^2$, and $\pi^{-1}\left(\left[\frac{31\pm 3\sqrt{-7}}{32}:1\right]\right)$ are a conjugate pair of nodal cubics. $S$ itself has seven isolated singular points, all of which are contained in the first three singular fibers (see Computation~\ref{comp:singularpoints}). 

In the same computation, we find that these singularities are all Du Val of type $A_n$ for $n\leq 3$. Du Val singularities can be resolved by a finite sequence of blowups at isolated double points. The result is a minimal smooth surface $\Sbar$ with a birational morphism $\varphi:\Sbar\to S$. The composition $\pibar = \pi\varphi$ gives an elliptic fibration of $\Sbar$ whose singular fibers are those of $S$ with each Du Val singularity replaced by a chain of rational curves. This is depicted in Figure~\ref{fig:ell}.
\begin{figure}[ht]
    \centering
    \includegraphics[width=1\linewidth,trim=4 4 4 4,clip]{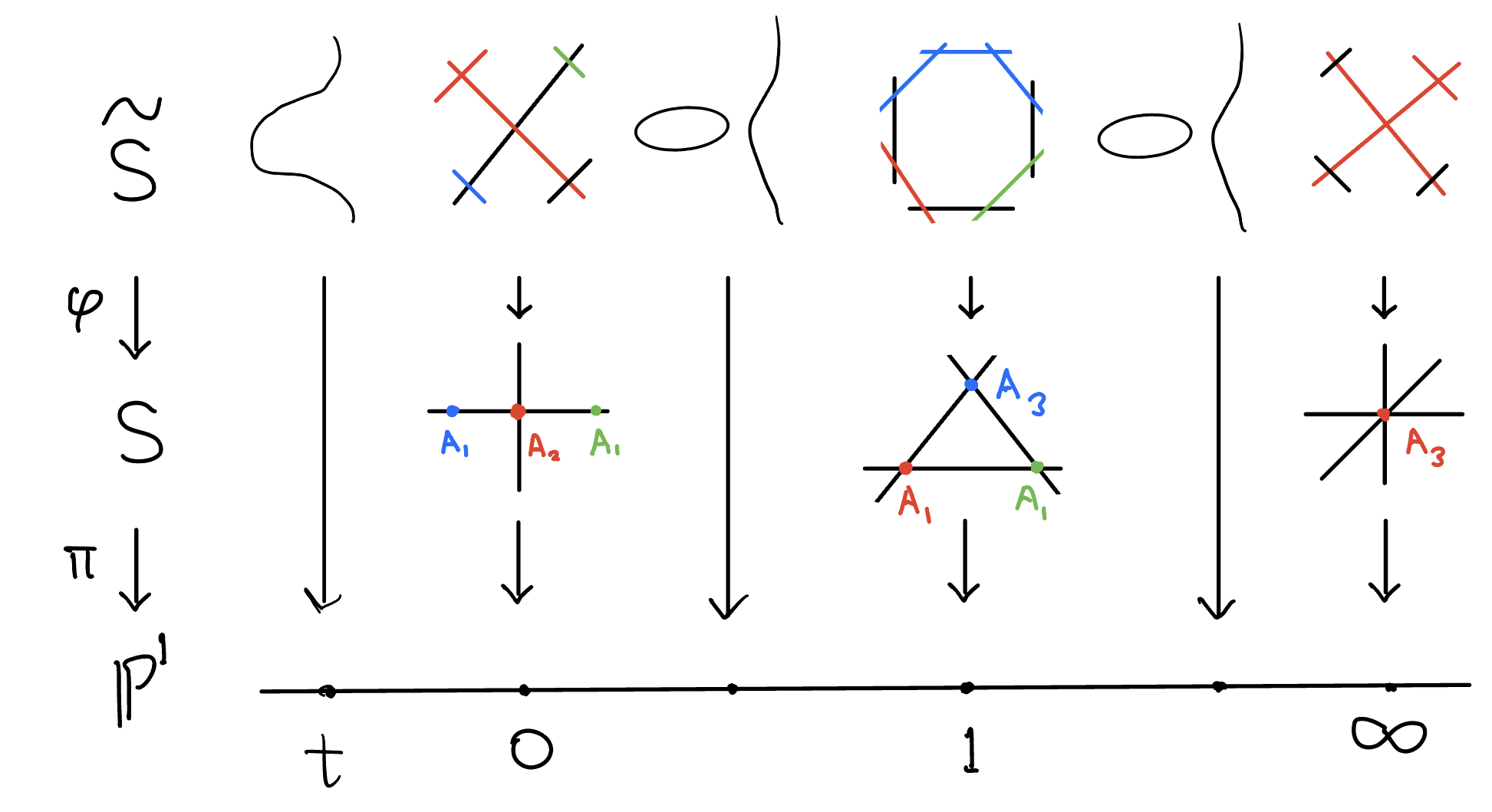}
    \caption{The elliptic fibrations of $S$ and $\Sbar$, drawn over $\RR$ with coordinate $t=u/v$ on $\PP^1$. The marked points are Du Val singularities of $S$ and are replaced by a chain of rational curves in $\Sbar$. Smooth fibers, as well as the nodal fibers (not pictured), are unchanged by $\varphi$.}
    \label{fig:ell}
\end{figure}

A smooth surface with a minimal elliptic fibration is called an \emph{elliptic surface} (see~\cite{schuttshioda}). All possible singular fibers of an elliptic surface were determined by Kodaira. In his notation, the nodal cubic fibers are of type $\mathrm{I}_1$. Inserting the appropriate trees of exceptional curves, we find that $\pibar^{-1}([0:1])$ and $\pibar^{-1}([1:0])$ are of type $\mathrm{I}_1^*$, while $\pibar^{-1}([1:1])$ is of type $\mathrm{I}_8$. (One can also compute the list of Kodaira fibers directly, which gives the same result; see Computation~\ref{comp:mordellweil}.) From this information, we obtain the topological Euler charateristic of $\Sbar(\CC)$:
$$e(\Sbar)=\sum_F e(F)=1+1+7+7+8=24$$
where the sum is over singular fibers \cite{schuttshioda}*{Theorem 6.10}. This is the correct Euler characteristic for a \emph{K3 surface}, i.e., a complete nonsingular surface $X$ with trivial canonical bundle $\omega_X\cong \mc O_X$ and irregularity $h^{1,0}(X)=h^1(X,\mc O_X)=0$.

\begin{lem}\label{lem:k3}
$\Sbar$ is a K3 surface.
\end{lem}

\begin{proof}
We first compute the canonical bundle $\omega_S$ on $S$. Recall that $S$ has bidegree $(2,3)$ in $\PP^1\times \PP^2$, while the canonical bundle $\omega_{\PP^1\times \PP^2}$ has bidegree $(-2,-3)$. $S$ is a regular in codimension $1$, so by the adjunction formula,
$$\omega_S=\left.(\omega_{\PP^1\times \PP^2}\otimes \mc O(S))\right\vert_S=\left.(\mc O(-2,-3)\otimes \mc O(2,3))\right\vert_S=\mc O_S.$$
We can also compute $h^1(S,\mc O_S)$ using the short exact sequence 
$$0\to \mathcal{O}(-2,-3) \to \mathcal{O}\to\mathcal{O}_S\to 0.$$
of sheaves on $\PP^1\times \PP^2$. The corresponding long exact sequence in cohomology is
$$\cdots \to H^1(\PP^1\times \PP^2,\mathcal{O})\to H^1(S,\mathcal{O}_{S})\to H^2(\PP^1\times \PP^2,\mathcal{O}(-2,-3))\to\cdots.$$
Since $h^{1,0}(\PP^1\times \PP^2)=h^{2,3}(\PP^1\times \PP^2)=0$, we have $h^1(S,\mc O_S)=0$ as well.

Since the singularities of $S$ are all Du Val, the resolution $\varphi$ is crepant  \cite{reid}, i.e., $$\omega_{\Sbar}=\varphi^*\omega_S=\varphi^*\mc O_S=\mc O_{\Sbar}.$$ Moreover, Du Val singularities are rational, meaning the natural map $$\mathcal{O}_S \to R\varphi_*\mathcal{O}_{\Sbar}$$ of complexes on $S$ is a quasi-isomorphism. Applying $R^1\Gamma$ yields 
$$h^1(\Sbar,\mc O_{\Sbar})=h^1(S,\mc O_S)=0,$$
so $\Sbar$ is a K3 surface. 
\end{proof}

\begin{rem} 
The resolution $\varphi$ is an isomorphism away from the singular points of $S$, which are disjoint from the open subset isomorphic to $\Rone$. It follows that $\Sbar$ also has such a subset, so it is the compactification of $\Rone$ by an elliptic K3 surface promised in Theorem~\ref{part:k3}.
\end{rem} 

We now compute enough standard invariants $\Sbar$ to determine it up to isomorphism. First is the \emph{Picard number} $\rho(\Sbar)$, the rank of the N\'eron--Severi group $\mathrm{NS}(\Sbar)$ of divisors modulo algebraic equivalence, which we compute using the Shioda--Tate formula.
\begin{lem}[\cite{schuttshioda}*{Theorem 6.3, Corollary 6.13}] \label{lem:shiodatate}
For any elliptic surface $X$ with identity section, we have $\mathrm{NS}(X)/A\cong \mathrm{MW}(X)$ where $A$ is the subgroup generated by the classes of the identity section and fiber components. Hence,
$$\rho(X)=2+\sum_F (m_F-1)+\operatorname{rank} \mathrm{MW}(X),$$
where $m_F$ is the number of components of the singular fiber $F$.
\end{lem}

\begin{lem}\label{lem:picard}
The Picard number $\rho(\Sbar)$ is $20$.
\end{lem}
\begin{proof}
In the proof of Proposition~\ref{prp:zariski}, we exhibited a non-torsion member of $\mathrm{MW}(S)\cong \mathrm{MW}(\Sbar)$, so its rank is at least $1$. By Lemma~\ref{lem:shiodatate}, 
$$\rho(\Sbar)\geq 2+(0+0+5+5+7)+1=20.$$
But $20$ is the largest possible Picard number for a K3 surface \cite{huybrechts}, so in fact $\operatorname{rank}\mathrm{MW}(\Sbar)=1$ and $\rho(\Sbar)=20$.
\end{proof}

It follows from the Torelli theorem for K3 surfaces that K3 surfaces $X$ of Picard number 20 are determined by their \emph{transcendental lattice} $T(X)$, the orthogonal complement of $\mathrm{NS}(X)$ in $H^2(X,\ZZ)$ \cites{huybrechts,schutt}. This is an even, positive-definite lattice of rank $2$. Following \cite{schutt}, we say that a K3 surface over $\QQ$ has \emph{Picard rank $20$ over $\QQ$} if $\rho(X)=20$ and $\mathrm{NS}(X)$ is generated by divisors defined over $\QQ$. Elkies showed that there are exactly $13$ such K3 surfaces, corresponding to the $13$ primitive lattices of class number $1$ \cites{elkies,schutt}. They are determined by the discriminant $d$ of $T(X)$, or equivalently (up to sign) the discriminant of $\mathrm{NS}(X)$.

\begin{lem}\label{lem:disc}
$\Sbar$ has Picard rank $20$ over $\QQ$ with discriminant $d=-7$.
\end{lem}
\begin{proof}
We use the Cox--Zucker formula \citelist{\cite{coxzucker}\cite{huybrechts}*{\S 11.3}}:
$$|d|=\frac{R}{\left|\mathrm{MW}(\Sbar)_{\mathrm{tors}}\right|^2}\prod_F n_F.$$
Here, $\mathrm{MW}(\Sbar)_{\mathrm{tors}}$ is the torsion subgroup of $\mathrm{MW}(\Sbar)$, $n_F$ is the number of components of the singular fiber $F$ appearing with multipilicity $1$, and $R$ is the \emph{regulator}, the discriminant of $\mathrm{MW}(\Sbar)/\mathrm{MW}(\Sbar)_{\mathrm{tors}}$ with respect to the height pairing $\langle -,-\rangle$ (see \cite{shioda}).

$\mathrm{I}_1$, $\mathrm{I}_1^*$, and $\mathrm{I}_8$ fibers have $n_F=1$, $4$, and $8$, respectively. In Computation~\ref{comp:mordellweil}, we find that the subgroup $H\subseteq\mathrm{MW}(\Sbar)_{\mathrm{tors}}$ consisting of sections defined over $\QQ$ is isomorphic to $\ZZ/4$, so $|\mathrm{MW}(\Sbar)_{\mathrm{tors}}|=4k$. Since $\mathrm{MW}(\Sbar)$ has rank $1$, $R$ is equal to $\langle g,g\rangle$ for a generator $g$ of $\mathrm{MW}(\Sbar)/\mathrm{MW}(\Sbar)_{\mathrm{tors}}$. We compute $\langle s,s\rangle=7/8$ for the section $s$ from the proof of Proposition ~\ref{prp:zariski}. Writing $s=ng$ for some $n\in \ZZ$, we have $R=7/8n^2$. Putting this all together, we find that 
$$|d|=\frac{7}{8n^2}\cdot\frac{1}{(4k)^2}\cdot 1\cdot 1\cdot 4\cdot 4\cdot 8=\frac{7}{k^2n^2}.$$
Since $T(\Sbar)$ is an even lattice of rank $2$, $d$ must be an integer congruent to $0$ or $1$ modulo $4$; we deduce that $k^2=n^2=1$ and $d=-7$. In particular, ${\mathrm{MW}(\Sbar)_{\mathrm{tors}}=H}$ and $s$ generates $\mathrm{MW}(\Sbar)/\mathrm{MW}(\Sbar)_{\mathrm{tors}}$, so the Mordell--Weil group is generated by sections defined over $\QQ$. Since the identity section and all components of the reducible fibers are also defined over $\QQ$, we conclude that $\Sbar$ has Picard rank $20$ over~$\QQ$.
\end{proof}

Comparing with the table of all K3 surfaces with Picard rank $20$ over $\QQ$ in \cite{schutt}*{\S 10}, we reach the remarkable conclusion that $\Sbar$ is isomorphic to the universal elliptic curve over $\Gamma_1(7)$. As a sanity check, we find our Kodaira fibers and Mordell--Weil group among the 20 elliptic fibrations of that modular surface \cite{lecacheux}*{Table 3, row 2}, and we verify in Computation~\ref{comp:mordellweil} that the two surfaces have isomorphic generic fibers.

Having fully analyzed $\widetilde{S}=\widetilde{\Sfc{V}}$, we turn to the other three configurations of interest.

\begin{proof}[Proof of Theorem~\ref{part:k3}] 
We follow the same steps to prove that the minimal resolutions $\widetilde{\Sfc{VIII}}$, $\widetilde{\Sfc{IX}}$, and $\widetilde{\Sfc{X}}$ are also K3 surfaces of Picard number $20$. The proof that they are K3 is identical to that of Lemma~\ref{lem:k3}. Since $\Sfc{VIII}$, $\Sfc{IX}$, and $\Sfc{X}$ have degree $(2,3)$ in $\PP^1\times \PP^2$, we need only check that they have only Du Val singularities; this is done in Computation~\ref{comp:k3}. We also compute the singular fibers of each surface, and check that $\sum_F (m_F-1)=17$. Since each Mordell--Weil group has an explicit non-torsion element, Lemma~\ref{lem:shiodatate} shows that the Picard number is $20$.

\begin{table}[htb]
    \renewcommand{\arraystretch}{1.5}
    \centering
    \begin{tabular}{|c|c|c|c|c|}
        \hline
        $X$ & Singular fibers & $\mathrm{MW}(X)$ & $\langle s,s\rangle$ & $d$ \\
        \hline\hline
        $\widetilde{\Sfc{V}}\cong \widetilde{\Sfc{IX}}$ & $2\mathrm{I}_1,2\mathrm{I}^*_1,\mathrm{I}_8$ & $\ZZ\oplus \ZZ/4$ & $7/8$ & $-7$ \\
        \hline
        $\widetilde{\Sfc{VIII}}$ & $\mathrm{I}^*_0,2\mathrm{I}^*_2,\mathrm{I}_2$ & $\ZZ\oplus \ZZ/2\oplus\ZZ/2$ & $1$ & $-8$ \\
        \hline
        $\widetilde{\Sfc{X}}$ & $3\mathrm{I}_1,\mathrm{I}_2,\mathrm{I}_5,\mathrm{I}_6,\mathrm{I}_8$ & $\ZZ\oplus \ZZ/2$ & $11/120$ & $-11$ \\
        \hline
    \end{tabular}
    \bigskip
    \caption{The data required to identify each K3 surface, collected in Computation \ref{comp:k3}: the singular fibers, the Mordell-Weil group, the height pairing $\langle s,s\rangle$, and the discriminant.}
    \label{tab:fib}
\end{table}

As in Lemma~\ref{lem:disc}, we can identify these K3 surfaces $X=\widetilde{\Sfc{VIII}}$, $\widetilde{\Sfc{IX}}$, and $\widetilde{\Sfc{X}}$ up to isomorphism by computing the group $H\subseteq MW(X)_{\mathrm{tors}}$ of torsion sections defined over $\QQ$, together with the height pairing $\langle s,s\rangle$ for some non-torsion section $s$ defined over $\QQ$. This is done in Computation~\ref{comp:k3}, with the results collected in Table~\ref{tab:fib}. In each case, the Cox--Zucker formula and the requirement that $d$ be an integer congruent to $0$ or $1$ modulo $4$ imply that $H=MW(X)_{\mathrm{tors}}$ and $s$ generates $MW(X)/MW(X)_{\mathrm{tors}}$ (so $R=\langle s,s\rangle$). Hence they all have Picard rank $20$ over $\QQ$, and so are determined up to isomorphism by their discriminant.
\end{proof}

Surprisingly, we find that $\widetilde{\Sfc{V}}$ and $\widetilde{\Sfc{IX}}$ are isomorphic over $\PP^1$, as checked explicitly in Computation~\ref{comp:k3}. We do not know how to interpret this isomorphism in terms of the (nonisomorphic) configurations $\L{V}$ and $\L{IX}$. One might suspect that they are related by projective duality (exchanging points and lines), but in fact all $10_3$ configurations are self-dual. If this isomorphism does arise from some combinatorial relationship, it is more subtle than this.

\section{Moduli space interpretation}\label{sec:mod}
Recall that the K3 surface $\Sbar$ has an open subset $R$ (its ``interior") which parameterizes the line arrangements realizing the configuration $\Lone$. We would like to extend this interpretation to the complement of the interior (the ``boundary"), which ought to parameterize ``degenerate" realizations where additional triples become collinear. We make this precise using the machinery of geometric invariant theory (GIT), which we briefly review.

In general, given an action of a reductive algebraic group $G$ on a projective variety $X$, a well-behaved quotient $X/G$ does not exist in the category of varieties. GIT gives a method for constructing a projective variety $X\git G$ which is the categorical quotient of a $G$-invariant open subset $X_{ss}\subset X$, called the \emph{semistable locus}; here, being a categorical quotient means that any $G$-invariant map $X_{ss}\to Y$ factors through the canonical map $X_{ss}\to X\git G$. The subset $X_{ss}$ depends on a choice of $G$-linearized ample line bundle on $X$; different choices yield different GIT quotients in general. This choice also determines an open subset $X_s\subseteq X_{ss}$, called the \emph{stable locus}, such that the image of $X_s$ in $X\git G$ is an honest geometric quotient $X_s/G$, in the sense that the fibers of $X_s\to X_s/G$ are $G$-orbits. This is summarized in the following diagram:
$$\begin{tikzcd}
    X_s\arrow[hook,r]\arrow[twoheadrightarrow,d] & X_{ss}\arrow[hook,r]\arrow[twoheadrightarrow,d]& X\\
    X_s/G\arrow[hook,r]& X\git G.
\end{tikzcd}$$
We will need the following fact:
\begin{lem}\label{lem:can}
    Suppose the action of $G$ on $X_{ss}$ is free and $X_s=X_{ss}$. Then the canonical bundle $\omega_Y$ of $Y=X\git G$ is trivial if and only if $\omega_{X_{ss}}$ is trivial as a $G$-linearized line bundle.
\end{lem}
\begin{proof}
    There is a $G$-equivariant short exact sequence
    \begin{equation}\label{eq:ses}
        0\to \rho^*\Omega_Y\to \Omega_{X_{ss}}\to \mathfrak{g}^\vee\otimes \mc O_{X_{ss}}\to 0
    \end{equation}
    where $\Omega_Y$ and $\Omega_{X_{ss}}$ denote the cotangent bundles on $Y$ and $X_{ss}$, $\rho:X_{ss}\to Y$ is the quotient map, and $\mathfrak{g}$ is the Lie algebra of $G$ with the adjoint representation (e.g., \cite{torres}*{\S 2.2}). Since $G$ acts trivially on the top exterior power of $\mathfrak g$, taking the top exterior power of \eqref{eq:ses} yields
    $$\rho^*\omega_Y\cong \omega_{X_{ss}}$$
    as $G$-linearized line bundles. Hence, if either bundle has a nowhere-vanishing $G$-invariant global section, the other does as well. Since $G$-invariant sections of $\rho^*\omega_Y$ are exactly sections of $\omega_Y$, the lemma is proved.
\end{proof}

For our purposes, we take $X=(\PP^2)^n$ for $n\geq 4$ and $G=\mathrm{PGL}(3)$. We refer to points of $(\PP^2)^n$ as \emph{arrangements} of $n$ points in $\PP^2$. Line bundles on $(\PP^2)^n$ are all of the form $$\mc O(d_1,\dots,d_n)=\pi_1^*\mc O(d_1)\otimes \cdots\otimes \pi_n^*\mc O(d_n)$$
where $\pi_i$ is the $i$-th projection onto $\PP^2$. These have a canonical $\mathrm{PGL}(3)$-linearization when $3$ divides $\sum_i d_i$ and are ample when $d_i>0$. It turns out that the semistable locus has a straightforward description in this case.
\begin{lem}[e.g., \cite{incensi}*{Proposition 1.1}] \label{lem:stab}
    Let $d=\sum_i d_i$ and $w_i=d_i/d$. Then $(p_1,\dots,p_n)$ is semistable if and only if
\begin{enumerate}
    \item for all $p\in \PP^2$, 
    $$\sum_{p_i=p}w_i\leq \frac{1}{3},$$
    and
    \item for all lines $\ell\subset \PP^2$,
    $$\sum_{p_i\in \ell}w_i\leq \frac{2}{3}.$$
\end{enumerate}
Stable points are characterized the same way, but with strict inequalities.
\end{lem}
We think of the $w_i$ as weights for the $n$ points, where an arrangement is unstable if too much weight is concentrated at a point or on a line. A choice of weights $w=(w_i)$ is called a \emph{weighting}, and the corresponding GIT quotient is denoted $\Q_w=(\PP^2)^n\git_w\mathrm{PGL}(3)$.

Two natural weightings come to mind. For the first, we designate $4$ of the $n$ points as ``heavy" and assign them weights close to $1/4$; the others are given nearly zero weight. We call this the \emph{oligarchic weighting}. Here, the sets of stable and semistable arrangements agree (there are no \emph{strictly semistable} points). Per Lemma~\ref{lem:stab}, an arrangement is stable exactly when no two of the four heavy points coincide and no three of them are collinear, with no restrictions on the other points. For this weighting, it is easy to see that the GIT quotient is isomorphic to $(\PP^2)^{n-4}$; we just fix the four heavy points to standard positions, and the others can be anywhere.

At the other extreme, we have the \emph{democratic weighting} $\delta$, where all weights $\delta_i$ are equal to $1/n$. Semistability now means that there are at most $n/3$ coincident points and at most $2n/3$ points on any line. Note that this is the same as stability unless $3$ divides $n$. The corresponding quotient $\Q_\delta$ is not as easy to describe as for the oligarchic weighting. The following lemma affords us a concrete characterization of the stable part $\Q_{w,s}$ of $\Q_w$ for any weighting $w$.
\begin{dfn}
Four points in $\PP^2$ are said to form a \emph{frame} if no three are collinear. We say that an arrangement $(p_1,\dots,p_n)\in (\PP^2)^n$ has a frame if some choice of four $p_i$ is a frame. The frame
$$f_1=[1:0:0],\quad f_2=[0:1:0],\quad f_3=[0:0:1],\quad f_4=[1:1:1]$$
is called the \emph{standard frame}.
\end{dfn}

\begin{lem}[\cite{keeltevelev}*{Lemma 8.6}] \label{lem:frame}
Every arrangement which is stable with respect to some weighting has a frame.
\end{lem}

\begin{proof}
It is clear that there are at least three non-collinear points in the arrangement, say $p_1,p_2,p_3$. If there is a fourth point not collinear with any two of $p_1,p_2,p_3$, then we're done, so suppose all other points lie on one of $\overline{p_1p_2}$, $\overline{p_1p_3}$, or $\overline{p_1p_3}$. Let $W_1$ be the combined weight of all points coincident with $p_1$, and similarly for $W_2$ and $W_3$. By stability, $W_i<1/3$. Since the sum of all weights is $1$ and $W_1+W_2+W_3<1$, there must be a point $p_4$ not coincident with $p_1$, $p_2$, or $p_3$. Suppose $p_4$ lies on $\overline{p_1p_2}$. Let $W_{12}<2/3$ be the combined weight of all points on $\overline{p_1p_2}$. Then $W_{12}+W_3<1$, so there must be a point $p_5$ neither on the line $\overline{p_1p_2}$ nor equal to $p_3$; say it lies on $\overline{p_2p_3}$. Then $p_1,p_3,p_4,p_5$ is a frame, as required.
\end{proof}

It is worth noting that there do exist strictly semistable arrangements that do not admit a frame. For example, when $3$ divides $n$, an arrangement with $n/3$ points at each vertex of a triangle is semistable with respect to the democratic weighting, but does not have a frame. 

\begin{cor}\label{cor:cover}
There exists an cover of $\Q_{w,s}$ by open subsets isomorphic to
$$U_{\i}=\left\{(p_i)_{i\not\in \i}\in(\PP^2)^{n-4}\mid (p_i)\in (\PP^2)^n_s\text{ where } p_{i_1}=f_1,\dots,p_{i_4}=f_4\right\}$$
indexed by $\i=\{i_1,\dots,i_4\}\subseteq\{0,\dots,n-1\}$ with $i_1<\cdots<i_4$. 

\end{cor}
\begin{proof}
 Suppose $(q_i)\in (\PP^2)^n_s$. By Lemma~\ref{lem:frame}, there are $q_{i_1},\dots,q_{i_4}$ that form a frame. Consider the $\mathrm{PGL}(3)$-invariant open set 
\begin{equation}\label{eq:cover}
    \left\{(p_i)\in (\PP^2)^n\mid p_{i_1},\dots,p_{i_4}\text{ form a frame}\right\}.
\end{equation}
As before, the quotient of this set by $\mathrm{PGL}(3)$ can be identified with $(\PP^2)^{n-4}$ by fixing $p_{i_1},\dots,p_{i_4}$ to be the standard frame. Intersecting \eqref{eq:cover} with the stable locus $(\PP^2)^n_s$ and passing to $\Q_{w,s}$ yields an open subset of $\Q_{w,s}$ containing the image of $(q_i)$ and isomorphic to $U_{\i}$. Hence these subsets cover $\Q_{w,s}$, as claimed.
\end{proof}

$\Q_{w,s}$ carries a universal $\PP^2$-bundle $B\to \Q_{w,s}$ with $n$ sections $P_0,\dots,P_{n-1}$ which are $w$-stable in each fiber $\PP^2$. Here, ``universal" means that any $\PP^2$-bundle on a variety $Y$ with $n$ fiberwise $w$-stable sections is the pullback of $B$ under a unique morphism $Y\to \Q_{w,s}$. We say that $\Q_{w,s}$ is a \emph{fine moduli space} for $w$-stable arrangements. The open cover $\{U_{\i}\}$ gives a local trivialization of $B$, where the sections $P_{i_1}=f_1,\dots,P_{i_4}=f_4$ are constant and the others are given by the $n-4$ projections $U_{\i}\hookrightarrow (\PP^2)^{n-4}\to \PP^2$. This is most interesting when there are no strictly $w$-semistable arrangements, so $\Q_{w,s}=\Q_w$ is the entire GIT quotient.

Consider now a configuration $\mc L$ with $n$ points, as in the introduction.
\begin{dfn}
An arrangement $(p_i)\in (\PP^2)^n$ is called a \emph{weak realization} of $\mc L$ if for all $\ell\in \mc L$ with $i,j,k\in \mc L$, the points $p_i,p_j,p_k$ are collinear.
\end{dfn}
The set of weak realizations forms a closed $\mathrm{PGL}(3)$-invariant subvariety $W$ of $(\PP^2)^n$ containing the set of realizations $V$ from the introduction as an open subset. This gives us a systematic method for furnishing compactifications of $\mc R(\mc L)$, provided that we choose a weighting $w$ such that all realizations of $\mc L$ are stable.
\begin{dfn}
The GIT quotient
$$\mc R_w(\mc L)= W\git_w\mathrm{PGL}(3)\subset \Q_w$$
is called the \emph{$w$-semistable realization space} of $\mc L$.
\end{dfn}
This is the coarse moduli space for $w$-semistable weak realizations of $\mc L$. When there are no strictly $w$-semistable arrangements, it is a fine moduli space with universal family pulled back from the one on $Q_w$. The closure of $\mc R(\mc L)$ in $\mc R_w(\mc L)$ is a compactification of $\mc R(\mc L)$ whose boundary points correspond to $w$-semistable degenerations of realizations of $\mc L$.

With these generalities in hand, we return to our $10_3$ configurations $\L{V}$, $\L{VIII}$, $\L{IX}$, and $\L{X}$. Seeing how $S$ was constructed in Section~\ref{sec:ell} by fixing four points, one might hope to identify $S$ or $\Sbar$ with the semistable realization space in the corresponding oligarchic quotient $(\PP^2)^6$. However, the rational map $S\dashrightarrow (\PP^2)^6$ sending a point in $S$ to its corresponding arrangement is not a morphism; its composition with $\Sbar\to S$ is, but this fails to be injective. The oligarchic realization space thus turns out to lie ``between" $S$ and $\Sbar$, so it does not grant the moduli interpretation we seek. 

In fact, it is the democratic weighting $\delta=\left(\frac{1}{10},\dots,\frac{1}{10}\right)$ which realizes $\Sbar$ and the other K3 surfaces. One easily checks that all realizations of any $10_3$ configuration are $\delta$-stable.

\begin{ntn} \label{ntn:cover}
Let $\mc L_N$ be any of $\L{V}$, $\L{VIII}$, $\L{IX}$, and $\L{X}$, and let $\widetilde{S_N}$ be the corresponding K3 surface from Theorem \ref{part:k3}. For $\i=\{i_1,\dots,i_4\}\subset \{0,\dots,9\}$ with $i_1<\cdots <i_4$, let $X_{\i}$ be the closed subset of $(\PP^2)^6$ corresponding to weak realizations $(p_i)\in(\PP^2)^{10}$ of $\mc L_N$ with $p_{i_1},\dots,p_{i_4}$ fixed to the standard frame. Let $X_{\i,s}=X_{\i}\cap U_{\i}$ be the open subset of $X_{\i}$ corresponding to $\delta$-stable arrangements. 
\end{ntn}

Since $3$ does not divide $10$, we have $\Q_\delta=\Q_{\delta,s}$. By Corollary~\ref{cor:cover}, the $X_{\i,s}$ form an open cover of $\mc R_\delta(\mc L_N)$. 
\begin{lem}
$\mc R(\mc L_N) \subseteq X_{\i,s}$ as subsets of $\mc R_\delta(\mc L_N)$.
\end{lem}

\begin{proof}
Suppose $X_{\i,s}$ is nonempty. Then there is an arrangement $(p_i)$ satisfying the collinearities in $\mc L_N$ (and possibly some not in $\mc L_N$) such that $p_{i_1},\dots,p_{i_4}$ form a frame. This means there is no line in $\mc L_N$ containing any three of those points. It follows that they form a frame in any realization of $\mc L_N$. Realizations of $\mc L_N$ are stable, so we have $\mc R(\mc L_N)\subseteq X_{\i,s}$, as desired. 
\end{proof}

\begin{lem}\label{lem:nonsing}
$\mc R_\delta(\mc L_N)$ is nonsingular and irreducible of dimension $2$.
\end{lem}

\begin{proof}
In Computation~\ref{comp:democratic}, we find that for every $\i$, $X_{\i,s}$ is either empty, or nonsingular and irreducible. Since the nonempty $X_{\i,s}$ all intersect, it follows that $\mc R_\delta(\mc L_N)$ is nonsingular and connected, hence irreducible. Since $\mc R_\delta(\mc L_N)$ contains $\mc R(\mc L_N)$ as an open subset and $\mc R(\mc L_N)$ has dimension $2$, $\mc R_\delta(\mc L_N)$ has dimension $2$ as well.
\end{proof}

\begin{proof}[Proof of Theorem~\ref{part:moduli}]
By Lemma~\ref{lem:frame}, $\mathrm{PGL}(3)$ acts freely on $(\PP^2)^{10}_{ss}$. From Lemma~\ref{lem:nonsing}, we see that the preimage $W=\rho^{-1}(\mc R_\delta(\mc L_N))$ under the quotient map $\rho:(\PP^2)^{10}_{ss}\to \Q_\delta$ is smooth of dimension $10$. $W$ is cut out by the $10$ equations defining the lines of $\mc L_N$, so it's a complete intersection of codimension $10$. By the argument given in the introduction, $W$ has (equivariantly) trivial canonical class; by Lemma~\ref{lem:can}, $\mc R_\delta(\mc L_N)$ also has trivial canonical class. Since $\widetilde{S_N}$ and $\mc R_\delta(\mc L_N)$ are birational (both being compactifications of $\mc R(\mc L_N)$), it follows that $\mc R_\delta(\mc L_N)$ has zero irregularity, so it is a K3 surface. Birational K3 surfaces are isomorphic, so in fact $\widetilde{S_N}\cong \mc R_\delta(\mc L_N)$.
\end{proof}

\bibliography{k3}

\newpage
\appendix

\section{Computations} \label{sec:comp}
The code used in this paper can be found at \url{https://github.com/eliassink/k3moduli}. 

\begin{comp}\label{comp:singularpoints} 
In the \href{https://github.com/eliassink/k3moduli/blob/main/magma/singularpoints}{magma/singularpoints} file, we compute the singular locus of $S=\Sfc{V}$. The result is a list of seven schemes representing the points
$$([1:1],[1:1:1]),\quad([1:1],[0:0:1]),\quad([1:1],[1:0:0]),$$
$$([0:1],[1:0:1]), \quad([0:1],[0:0:1]), \quad([0:1],[1:0:0]),$$
$$([1:0],[0:1:0]).$$
None of these points lie in $\Rone\subset S$, as the corresponding arrangements have unwanted collinearities; the first three have $1,5,7$ collinear, the middle three have $1,3,7$ collinear, and the last has $1,2,7$ collinear. We check that these singularities are all Du Val and compute their resolution graphs (all $A_n$ for $n\leq 3$; compare Figure~\ref{fig:ell}). We also find the singular points of fibers of $\pi:S\to \PP^1$, which gives the list of singular fibers in Section~\ref{sec:k3}.
\end{comp}
\begin{comp}\label{comp:mordellweil}
In the \href{https://github.com/eliassink/k3moduli/blob/main/magma/mordellweil}{magma/mordellweil} file, we perform the computations in the Mordell-Weil group $\mathrm{MW}(S)=\mathrm{MW}(\Sbar)$ needed for Sections~\ref{sec:ell} and \ref{sec:k3}. In particular, we compute the Weierstrass form for the generic fiber $E$ and show that the section $s$ from the proof of Proposition~\ref{prp:zariski} is not torsion. We also directly compute the Kodaira fibers, the torsion subgroup, and the height pairing $\langle s,s\rangle$ needed to prove Lemmas \ref{lem:picard} and \ref{lem:disc}. Finally, we check that our elliptic fibration is isomorphic to the elliptic fibration given in \cite{lecacheux}*{Table 3, row 2}.
\end{comp}
\begin{comp}\label{comp:other}
In the \href{https://github.com/eliassink/k3moduli/blob/main/magma/configurations/other}{magma/configurations/other} folder, we give construction sequences for the configurations which do not yield K3 surfaces. As described in the proof of Theorem \ref{thm:dense}, we find that the resulting equation is either trivial or reducible into rational components, and moreover that the realization space is contained in at most one such component. 
\end{comp}
\begin{comp}\label{comp:k3}
In the \href{https://github.com/eliassink/k3moduli/blob/main/magma/configurations/k3}{magma/configurations/k3} folder, we give construction sequences for $\L{VIII}$, $\L{IX}$, and $\L{X}$ and repeat Computations~\ref{comp:singularpoints} and \ref{comp:mordellweil} for these configurations. For $\L{IX}$, we check explicitly that $\widetilde{\Sfc{V}}$ and $\widetilde{\Sfc{IX}}$ have isomorphic elliptic fibrations. For $\L{X}$, we find that there are not enough sections to prove analytic density.
\end{comp}

\begin{comp}\label{comp:democratic}
In the \href{https://github.com/eliassink/k3moduli/blob/main/macaulay2/democratic}{macaulay2/democratic} file, we show that for every $\i$, $X_{\i}$ is irreducible and $X_{\i,s}$ is nonsingular (see Notation~\ref{ntn:cover}). More precisely, we show that the singular locus of $X_{\i}$ is contained in the unstable locus $X_{\i}\smallsetminus X_{\i,s}$. This latter computation is performed in affine charts and takes several hours per configuration to complete. 
\end{comp}
\end{document}